\documentclass[10pt,reqno]{amsart}%
\usepackage{mathtools}
\usepackage{mathptmx}
\usepackage{geometry}
\usepackage{enumitem}
\usepackage[expansion=false]{microtype}
\usepackage{xcolor}
\usepackage{hyperref}
\usepackage{amsmath}
\usepackage{amsfonts}
\usepackage{amssymb}
\usepackage{amsthm}
\hypersetup{
colorlinks=true,
linkcolor={blue!60!black},
citecolor={blue!60!black},
urlcolor={blue!80!black},
}

\theoremstyle{plain}
\newtheorem{theorem}{Theorem}[section]

\newtheorem{lemma}[theorem]{Lemma}
\newtheorem{corollary}[theorem]{Corollary}
\theoremstyle{definition}

\newtheorem{example}[theorem]{Example}
\newtheorem{remark}[theorem]{Remark}
\DeclareMathOperator{\Image}{Im}
\DeclareMathOperator{\Ker}{ker}
\begin{document}
\title[An integral formula for the inhomogeneous JvN equation]{An integral formula for the inhomogeneous Jordan--von Neumann equation}

\author{Alexandra Paicu}
\address{Department of Mathematics, Technical University of Cluj-Napoca,
         28 Memorandumului Street, 400114 Cluj-Napoca, Romania}
\email{alexandra.paicu@math.utcluj.ro}

\author{Dorian Popa}
\address{Department of Mathematics, Technical University of Cluj-Napoca,
         28 Memorandumului Street, 400114 Cluj-Napoca, Romania.\newline
         ORCID: \href{https://orcid.org/0000-0001-6197-3138}{0000-0001-6197-3138}}
\email{popa.dorian@math.utcluj.ro}

\author{Mircea Dan Rus}
\address{Department of Mathematics, Technical University of Cluj-Napoca,
         28 Memorandumului Street, 400114 Cluj-Napoca, Romania.\newline
         ORCID: \href{https://orcid.org/0000-0001-6255-0241}{0000-0001-6255-0241}}
\email{rus.mircea@math.utcluj.ro}

\date{\today}

\begin{abstract}
We study the inhomogeneous form of the Jordan--von Neumann quadratic
functional equation, in which the right-hand side is a prescribed function $g
$ of two real variables. We prove that the existence of a $C^{2}$ solution is
equivalent to $g$ being itself of class $C^{2}$ and satisfying a single
three-variable cocycle identity, and we exhibit the solution as a closed-form
integral expression involving the second partial derivative of $g $ along the
first coordinate axis. The construction preserves regularity along the
standard scale of $C^{k}$, smooth, and polynomial classes.

\end{abstract}
\maketitle

\section{Introduction}

\label{sec:intro}

The Jordan--von Neumann functional equation
\begin{equation}
\label{eq:JvN}f(x + y) + f(x - y) = 2 f(x) + 2 f(y),\qquad x, y\in\mathbb{R},
\end{equation}
is the algebraic skeleton of the parallelogram identity and, by the celebrated
theorem of Jordan and von Neumann~\cite{JordanVonNeumann1935}, characterizes
those norms on a real vector space that arise from an inner product. Its
solution space admits an explicit description: by
Kurepa~\cite{Kurepa1956,Kurepa1964}, every $f\colon\mathbb{R}\to\mathbb{R}$
satisfying~(\ref{eq:JvN}) is of the form $f(x) = B(x, x)$ for a unique
$\mathbb{Q}$-bilinear symmetric form $B\colon\mathbb{R}\times\mathbb{R}%
\to\mathbb{R}$, and under any of several mild regularity hypotheses (e.g.,
continuity at one point, Lebesgue measurability, or boundedness on a set of
positive measure), this collapses to $f(x) = c x^{2}$ for some $c\in
\mathbb{R}$. Comprehensive accounts of this theory are given in the monographs
of Acz\'el and Dhombres~\cite{AczelDhombres1989} and Kuczma~\cite{Kuczma2009}.

In the present paper we are concerned with the \emph{inhomogeneous} version
of~(\ref{eq:JvN}): for a prescribed function $g\colon\mathbb{R}^{2}%
\to\mathbb{R}$, we ask when there exists $f\colon\mathbb{R}\to\mathbb{R}$ such
that
\begin{equation}
\label{eq:inhom}f(x + y) + f(x - y) - 2 f(x) - 2 f(y) = g(x, y),\qquad(x,
y)\in\mathbb{R}^{2},
\end{equation}
and how $f$ is then determined by $g$. We denote by $\mathsf{Q}$ the linear
operator $f\mapsto\mathsf{Q}[f]$ defined by the left-hand side
of~(\ref{eq:inhom}), which now reads $\mathsf{Q}[f] = g$; the kernel
$\Ker\mathsf{Q}$ is exactly the solution space of~(\ref{eq:JvN}).

The question is natural from several perspectives. From the perspective of the
stability theory of functional equations initiated by Hyers and Ulam---and
developed for the quadratic equation by Skof~\cite{Skof1983},
Cholewa~\cite{Cholewa1984}, and Czerwik~\cite{Czerwik1992}; see
also~\cite{HyersIsacRassias1998}---the data $g$ represents a controlled
perturbation of~(\ref{eq:JvN}), and the recovery of $f$ from $g$ is the exact
inverse of the perturbation map. From an algebraic standpoint, the question
parallels the classical inverse problem for the linear Cauchy operator
$\mathsf{L}[f](x, y) := f(x + y) - f(x) - f(y)$. For~$\mathsf{L}$, it is a
standard fact that $g$ belongs to the image of $\mathsf{L}$ if and only if $g$
satisfies both the cocycle identity
\begin{equation}
\label{eq:Lcocycle}g(x, y) + g(x + y, z) = g(x, y + z) + g(y, z),\qquad x, y,
z\in\mathbb{R},
\end{equation}
and the symmetry condition $g(x, y) = g(y, x)$; identity~(\ref{eq:Lcocycle})
is the standard $2$-cocycle equation that arises in group cohomology
\textup{(}see, e.g., \cite{AczelDhombres1989}\textup{)} and is forced by the
associativity of addition. Erd\H{o}s~\cite{Erdos1959} showed that the
situation is sensitive to regularity: without any regularity hypothesis, the
symmetry condition does not follow from the cocycle identity, as he
established by exhibiting an asymmetric solution of~(\ref{eq:Lcocycle}) via a
Hamel-basis construction; under continuity of $g$, however, the cocycle
identity~(\ref{eq:Lcocycle}) implies the symmetry of $g$, and every continuous
$g$ satisfying~(\ref{eq:Lcocycle}) is then in the image of $\mathsf{L}$. For
$g\in C^{1}(\mathbb{R}^{2})$, Prunescu~\cite{Prunescu2007} gave a constructive
counterpart that exhibits an explicit solution.

We refer to~(\ref{eq:Lcocycle}) as the cocycle identity for $\mathsf{L}$. The
analog for the quadratic operator $\mathsf{Q}$ is the three-variable identity
\begin{equation}
\label{eq:cocycle}g(x + y, z) + g(x - y, z) + 2 g(x, y) = g(x, y + z) + g(x, y
- z) + 2 g(y, z),
\end{equation}
holding for all $x, y, z\in\mathbb{R}$ whenever $g\in\Image\mathsf{Q}$. The
derivation is parallel to that of~(\ref{eq:Lcocycle}) for $\mathsf{L}$ but
relies on the parallelogram structure of $\mathsf{Q}$ rather than on the
associativity of addition; see Section~\ref{sec:cocycle}. By analogy with the
linear case, we call identity~(\ref{eq:cocycle}) the cocycle identity for
$\mathsf{Q}$.

Our main result, stated below, asserts that, in the class $C^{2}$,
identity~(\ref{eq:cocycle}) is both necessary and sufficient for $g$ to lie in
the image of $\mathsf{Q}$, and that a solution is then given by a closed-form integral.

\begin{theorem}
\label{thm:main} Let $g\colon\mathbb{R}^{2}\to\mathbb{R}$.

\begin{enumerate}
[label=\textup{(\roman*)},leftmargin=*]

\item \label{it:nec} If $g = \mathsf{Q}[f]$ for some $f\colon\mathbb{R}%
\to\mathbb{R}$, then $g$ satisfies the cocycle identity~(\ref{eq:cocycle}).

\item \label{it:suf} Conversely, if $g\in C^{2}(\mathbb{R}^{2})$
satisfies~(\ref{eq:cocycle}), then the function $f\colon\mathbb{R}%
\to\mathbb{R}$ defined by
\begin{equation}
\label{eq:integral}f(x) = -\tfrac{1}{2}\, g(0, 0) - \tfrac{1}{2}\,
(\partial_{2} g)(0, 0)\, x + \tfrac{1}{2}\int_{0}^{x} (x - t)\, (\partial
_{2}^{2} g)(t, 0)\, \mathrm{d} t
\end{equation}
belongs to $C^{2}(\mathbb{R})$ and satisfies $\mathsf{Q}[f] = g$ on
$\mathbb{R}^{2}$.
\end{enumerate}

The solution is unique up to addition of an element of $\Ker\mathsf{Q}$.
\end{theorem}

Throughout, $\partial_{1}$ and $\partial_{2}$ denote partial differentiation
in the first and second variables, respectively. The argument for
part~\ref{it:suf} proceeds by differentiating the cocycle
identity~(\ref{eq:cocycle}) in $z$ at $z = 0$. One differentiation yields that
$(\partial_{2} g)(\cdot, 0)$ is constant on $\mathbb{R}$; two differentiations
produce an identity satisfied by $h(t) := (\partial_{2}^{2} g)(t, 0)$ that
algebraically reproduces $\mathsf{Q}[f]$ for any $f$ with $f^{\prime\prime}=
h/2$. Two integrations then give the formula~(\ref{eq:integral}). The
construction is the analog for $\mathsf{Q}$ of the integral formula obtained
by Prunescu~\cite{Prunescu2007} for $\mathsf{L}$ in the class $C^{1}$.

\newpage

\section{The cocycle identity and its consequences}

\label{sec:cocycle}

\subsection{Necessity}

\begin{proof}
[Proof of Theorem~\textnormal{\ref{thm:main}~\ref{it:nec}}]Set $g = \mathsf{Q}%
[f]$. We show that both sides of~(\ref{eq:cocycle}) expand to
\begin{multline}
\label{eq:Sdef}S(x, y, z) := f(x + y + z) + f(x + y - z) + f(x - y + z) + f(x
- y - z)\\
- 4 f(x) - 4 f(y) - 4 f(z).
\end{multline}
For the left-hand side, sum
\begin{align*}
g(x + y, z)  &  = f(x + y + z) + f(x + y - z) - 2 f(x + y) - 2 f(z),\\
g(x - y, z)  &  = f(x - y + z) + f(x - y - z) - 2 f(x - y) - 2 f(z),\\
2 g(x, y)  &  = 2 f(x + y) + 2 f(x - y) - 4 f(x) - 4 f(y);
\end{align*}
the $\pm2 f(x + y)$ and $\pm2 f(x - y)$ terms cancel, and the result is $S(x,
y, z)$. For the right-hand side, the analogous expansion of $g(x, y + z) +
g(x, y - z) + 2 g(y, z)$ produces the same expression $S(x, y, z)$.
\end{proof}

\subsection{Operatorial reading of the cocycle identity}

\label{subsec:origin}

The cocycle identities~(\ref{eq:Lcocycle}) and (\ref{eq:cocycle}) admit a
unifying operatorial description that places them in a general framework. Let
$T\colon\mathcal{F}(\mathbb{R}, \mathbb{R})\to\mathcal{F}(\mathbb{R}^{2},
\mathbb{R})$ be a linear operator measuring the defect of some property
compatible with the group structure of $(\mathbb{R}, +)$, in the sense that
$T[f]\equiv0$ characterizes that property (here $\mathcal{F}(\mathbb{R}^{n},
\mathbb{R})$ denotes the space of real-valued functions on $\mathbb{R}^{n}$).
The iterated application of $T$ along the two coordinate axes of $T[f]$
produces two three-variable defects: at $z$ fixed, the slice $t\mapsto T[f](t,
z)$ has its own defect $T\bigl[T[f](\cdot, z)\bigr](x, y)$; at $x$ fixed, the
slice $t\mapsto T[f](x, t)$ has the defect $T\bigl[T[f](x, \cdot)\bigr](y,
z)$. The \emph{iterated cocycle identity}
\begin{equation}
\label{eq:iterated-cocycle}T\bigl[\, T[f](\,\cdot\,, z)\,\bigr](x, y) \;=\;
T\bigl[\, T[f](x, \,\cdot\,)\,\bigr](y, z), \qquad x, y, z\in\mathbb{R},
\end{equation}
asserts that these two iterated defects coincide: the iterated defect of
$T[f]$ is axis-symmetric. This is the precise content of the cocycle condition
for any such $T$, and it specializes to the standard $2$-cocycle identity for
$T = \mathsf{L}$ as well as to~(\ref{eq:cocycle}) for $T = \mathsf{Q}$.

\subsection{Consequences of the cocycle identity}

\begin{lemma}
\label{lem:consequences}Suppose $g\colon\mathbb{R}^{2}\rightarrow\mathbb{R} $
satisfies the cocycle identity~(\ref{eq:cocycle}). Then:

\begin{enumerate}
[label=\textup{(\roman*)},leftmargin=*]

\item \label{it:axis} $g(x, 0) = g(0, 0)$ for every $x\in\mathbb{R}$;

\item \label{it:first} if $g\in C^{1}(\mathbb{R}^{2})$, then $(\partial_{2}
g)(x, 0) = (\partial_{2} g)(0, 0)$ for every $x\in\mathbb{R}$;

\item \label{it:second} if $g\in C^{2}(\mathbb{R}^{2})$, then, with $h(t) =
(\partial_{2}^{2} g)(t, 0)$,
\begin{equation}
\label{eq:secondorder}h(x + y) + h(x - y) - 2 h(y) = 2 (\partial_{2}^{2} g)(x,
y),\qquad(x, y)\in\mathbb{R}^{2}.
\end{equation}

\end{enumerate}
\end{lemma}

\begin{proof}
\ref{it:axis} Set $y = 0$ and $z = 0$ in~(\ref{eq:cocycle}) and obtain $4 g(x,
0) = 2 g(x, 0) + 2 g(0, 0)$, so $g(x, 0) = g(0, 0)$.\smallskip

\ref{it:first} Differentiate~(\ref{eq:cocycle}) in $z$ at $z = 0$. By the
chain rule, the terms $g(x, y \pm z)$ on the right-hand side contribute
$\pm(\partial_{2} g)(x, y)$ and cancel, while $g(x \pm y, z)$ and $g(y, z)$
contribute, respectively, $(\partial_{2} g)(x \pm y, 0)$ and $(\partial_{2}
g)(y, 0)$. The resulting identity reads
\begin{equation}
\label{eq:axial-first}(\partial_{2} g)(x + y, 0) + (\partial_{2} g)(x - y, 0)
= 2 (\partial_{2} g)(y, 0).
\end{equation}
Setting $y = 0$ in~(\ref{eq:axial-first}) gives the claim.\smallskip

\ref{it:second} Differentiate~(\ref{eq:cocycle}) twice in $z$ at $z = 0$. By
the chain rule, each term $g(x, y \pm z)$ contributes $(\partial_{2}^{2} g)(x,
y) $ on the right-hand side; the remaining terms produce $h(x \pm y)$ on the
left and $2 h(y)$ on the right.
\end{proof}

\section{The integral formula}

\label{sec:formula}

\begin{proof}
[Proof of Theorem~\textnormal{\ref{thm:main}~\ref{it:suf}}]Let $h(t)=(\partial
_{2}^{2}g)(t,0)$ and $f$ be defined by~(\ref{eq:integral}), i.e.,%
\[
f(x)=-\tfrac{1}{2}\,g(0,0)-\tfrac{1}{2}\,(\partial_{2}g)(0,0)\,x+\tfrac{1}%
{2}\int_{0}^{x}(x-t)\,h(t)\,\mathrm{d}t\text{.}%
\]
Since $h\in C(\mathbb{R})$, the function $f$ is well-defined and belongs to
$C^{2}(\mathbb{R})$, with
\begin{equation}
f(0)=-\tfrac{1}{2}\,g(0,0),\qquad f^{\prime}(0)=-\tfrac{1}{2}\,(\partial
_{2}g)(0,0),\qquad f^{\prime\prime}(x)=\tfrac{1}{2}\,h(x). \label{eq:fderivs}%
\end{equation}
Set $E(x,y)=\mathsf{Q}[f](x,y)-g(x,y)$. By linearity of $\partial_{2}^{2}$ and
the definition of $\mathsf{Q}$,
\begin{align*}
(\partial_{2}^{2}\mathsf{Q}[f])(x,y)  &  =f^{\prime\prime}(x+y)+f^{\prime
\prime}(x-y)-2f^{\prime\prime}(y)\\
&  =\tfrac{1}{2}\bigl[h(x+y)+h(x-y)-2h(y)\bigr],
\end{align*}
which equals $(\partial_{2}^{2}g)(x,y)$ by Lemma~\ref{lem:consequences}%
~\ref{it:second}. Hence, $\partial_{2}^{2}E\equiv0$, so
\[
E(x,y)=A(x)+y\,B(x)
\]
for some $A,B\colon\mathbb{R}\rightarrow\mathbb{R}$. By
Lemma~\ref{lem:consequences}~\ref{it:axis}, $g(x,0)=g(0,0)$, and
by~(\ref{eq:fderivs}), $\mathsf{Q}[f](x,0)=-2f(0)=g(0,0)$. Hence,
$A(x)=E(x,0)=0$. Differentiating $E(x,y)=yB(x)$ in $y$ gives $B(x)=(\partial
_{2}E)(x,0)=(\partial_{2}\mathsf{Q}[f])(x,0)-(\partial_{2}g)(x,0)$. We
compute
\[
(\partial_{2}\mathsf{Q}[f])(x,0)=f^{\prime}(x)-f^{\prime}(x)-2f^{\prime
}(0)=-2f^{\prime}(0)=(\partial_{2}g)(0,0)
\]
by~(\ref{eq:fderivs}), and $(\partial_{2}g)(x,0)=(\partial_{2}g)(0,0)$ by
Lemma~\ref{lem:consequences}~\ref{it:first}. Hence, $B\equiv0$, $E\equiv0$,
and $\mathsf{Q}[f]=g$.

Uniqueness modulo $\Ker\mathsf{Q}$ is the standard consequence of the
linearity of $\mathsf{Q}$.
\end{proof}

\begin{remark}
\label{rem:forms} The representation~(\ref{eq:integral}) can also be written
as an iterated double integral:
\begin{equation}
\label{eq:iterated}f(x) = -\tfrac{1}{2}\, g(0, 0) - \tfrac{1}{2}\,
(\partial_{2} g)(0, 0)\, x + \tfrac{1}{2}\int_{0}^{x}\!\!\int_{0}^{v}\!\!
(\partial_{2}^{2} g)(t, 0)\, \mathrm{d} t\, \mathrm{d} v;
\end{equation}
the equivalence with~(\ref{eq:integral}) follows by integration by parts or by
Fubini's theorem.
\end{remark}

\begin{corollary}
\label{cor:regularity} Suppose $g$ satisfies the cocycle
identity~(\ref{eq:cocycle}) and belongs to one of the regularity classes
$C^{k}(\mathbb{R}^{2}) $ \textup{(}with $k\ge2$\textup{)}, $C^{\infty
}(\mathbb{R}^{2})$, or the class of real polynomials in two variables. Then
the function $f$ defined by~(\ref{eq:integral}) belongs to the corresponding
class $C^{k}(\mathbb{R})$, $C^{\infty}(\mathbb{R})$, or the class of real
polynomials in one variable, respectively.
\end{corollary}

\begin{proof}
The function $h = (\partial_{2}^{2} g)(\cdot, 0)$ lies in the corresponding
one-variable regularity class, hence so does $f^{\prime\prime}= h/2$, and
hence so does $f$ by the iterated-integral form~(\ref{eq:iterated}). If $g$ is
a real polynomial in two variables of total degree $d$, then $h$ is a real
polynomial in one variable of degree at most $d$, and $f$ is a real polynomial
in one variable of degree at most $d + 2$.
\end{proof}

\section{Examples and further directions}

We illustrate and verify Theorem~\ref{thm:main} with some examples of
functions $g$ that satisfy the cocycle identity~(\ref{eq:cocycle}) and the
corresponding solutions.

\label{sec:examples}

\begin{example}
\label{ex:quartic} Take $g(x,y)=x^{2}y^{2}$, $(x,y)\in\mathbb{R}^{2}$, which
satisfies~(\ref{eq:cocycle}) (both sides equal
 $2(x^{2}y^{2}+x^{2}z^{2}+y^{2}z^{2})$), hence Theorem~\ref{thm:main} applies. With $g(0,0)=0$, $(\partial
_{2}g)(0,0)=0$, $h(t)=(\partial_{2}^{2}g)(t,0)=2\,t^{2}$,
(\ref{eq:integral}) yields
\[
f(x)=\tfrac{1}{2}\int_{0}^{x}(x-t)\cdot2\,t^{2}\,\mathrm{d}t=\tfrac{x^{4}}%
{3}-\tfrac{x^{4}}{4}=\tfrac{x^{4}}{12},
\]
so $g=\mathsf{Q}\left[\tfrac{x^{4}}{12}\right]$.
\end{example}

\begin{example}
\label{ex:exp} Take $g(x,y)=2\,\mathrm{e}^{x}\cosh y-2\,\mathrm{e}%
^{x}-2\,\mathrm{e}^{y}$, $(x,y)\in\mathbb{R}^{2}$, which
satisfies~(\ref{eq:cocycle}) since $g=\mathsf{Q}[\mathrm{e}^{x}]$. Then
$g(0,0)=-2$, $(\partial_{2}g)(0,0)=-2$, $h(t)=2\,\mathrm{e}^{t}-2$, and
~(\ref{eq:integral}) yields
\begin{align*}
f(x) &  =1+x+\tfrac{1}{2}\int_{0}^{x}(x-t)\,(2\,\mathrm{e}^{t}-2)\,\mathrm{d}%
t\\
&  =1+x+(\mathrm{e}^{x}-x-1)-\tfrac{x^{2}}{2}=\mathrm{e}^{x}-\tfrac{x^{2}}{2},
\end{align*}
which is $\mathrm{e}^{x}$ modulo the kernel element $-\tfrac{1}{2}x^{2}%
\in\Ker\mathsf{Q}$.
\end{example}

\begin{example}
\label{ex:cos} Take $g(x,y)=2\cos x\cos y-2\cos x-2\cos y$, $(x,y)\in
\mathbb{R}^{2}$, which satisfies~(\ref{eq:cocycle}) since $g=\mathsf{Q}[\cos
x]$. Then $g(0,0)=-2$, $(\partial_{2}g)(0,0)=0$, $h(t)=2-2\cos t$, and
~(\ref{eq:integral}) yields
\[
f(x)=1+\tfrac{1}{2}\int_{0}^{x}(x-t)\,(2-2\cos t)\,\mathrm{d}t=1+\tfrac{x^{2}%
}{2}+(\cos x-1)=\cos x+\tfrac{x^{2}}{2},
\]
which is $\cos x$ modulo the kernel element $\tfrac{1}{2}x^{2}\in
\Ker\mathsf{Q}$.
\end{example}

\subsection*{Further directions}

\emph{The purely algebraic problem.} Whether the cocycle
identity~(\ref{eq:cocycle}) alone, with no regularity hypothesis on $g$, is
sufficient to ensure $g\in\Image\mathsf{Q}$, is open. The corresponding
question for the linear operator $\mathsf{L}$ has a definitive answer:
Erd\H{o}s~\cite{Erdos1959} showed that the linear cocycle
identity~(\ref{eq:Lcocycle}) must be supplemented by a symmetry condition
$g(x, y) = g(y, x)$, without which Hamel-basis constructions yield solutions
of~(\ref{eq:Lcocycle}) that are not in the image of~$\mathsf{L}$. A parallel
obstruction may exist for $\mathsf{Q}$ and deserves separate study.

\emph{The continuous case.} For the linear operator $\mathsf{L}$,
Erd\H{o}s~\cite{Erdos1959} proved that every continuous solution of the
cocycle identity~(\ref{eq:Lcocycle}) is automatically symmetric, hence lies in
the image of $\mathsf{L}$. Whether the cocycle identity~(\ref{eq:cocycle}) for
$\mathsf{Q}$ admits an analogous conclusion under mere continuity of
$g$---without the $C^{2}$ hypothesis used here---is an attractive open question.

\emph{The $C^{1}$ case.} The formula~(\ref{eq:integral}) requires $g\in
C^{2}(\mathbb{R}^{2})$ in order to define $h = (\partial_{2}^{2} g)(\cdot,
0)$. For $g\in C^{1}(\mathbb{R}^{2})$ satisfying~(\ref{eq:cocycle}), only one
differentiation of the cocycle is available, and a closed-form solution would
have to rely on a different combination of $g$, $\partial_{1} g$ and
$\partial_{2} g$. This would parallel the $C^{1}$ formula obtained by
Prunescu~\cite{Prunescu2007} for $\mathsf{L}$ and seems worth pursuing.

\bibliographystyle{spmpsci}
\bibliography{integral-formula-quadratic}

\end{document}